\documentclass[a4paper,11pt,twoside,reqno]{article}
\usepackage{amsmath, amsfonts, amssymb, amsthm}
\usepackage{stmaryrd}
\usepackage{bbm}
\usepackage{cases}
\usepackage{hyperref}
\usepackage[T1]{fontenc}
\usepackage[latin1]{inputenc}
\usepackage[english]{babel}
\usepackage[cyr]{aeguill}
\usepackage{xspace}
\usepackage{graphicx}
\usepackage{caption}
\usepackage{color}
\usepackage[labelformat=empty]{caption}
\usepackage{epigraph}
\usepackage{etoolbox}
\usepackage{enumitem}
\usepackage{stmaryrd}
\usepackage[margin=1in]{geometry}
\usepackage{cite}
\usepackage{upgreek}
\usepackage{dsfont}

\DeclareMathOperator*{\inflimit}{\underline{\lim}}
\DeclareMathOperator*{\esssup}{ess\,sup}
\DeclareMathOperator{\Tr}{Tr}

\theoremstyle{plain}
\newtheorem{thm}{Theorem}[section]
\newtheorem{lem}[thm]{Lemma}
\newtheorem{prop}[thm]{Proposition}

\newtheorem{assum}{Assumption}
\newtheorem*{note}{Notation}

\theoremstyle{definition}
\newtheorem{defn}{Definition}[section]

\newtheorem{exmp}{Example}[section]

\theoremstyle{remark}
\newtheorem{rem}{Remark}

\begin{document}
\title{\textbf{\huge On path-dependent multidimensional forward-backward SDEs}}
\author{Kaitong HU \footnote{CMAP, Ecole Polytechnique, hukaitong@gmail.com.}
             \and
             Zhenjie Ren \footnote{CEREMADE, Universit\'e Paris-Dauphine, PSL Research University. ren@ceremade.dauphine.fr.} 
	     \and
	     Nizar Touzi \footnote{CMAP, Ecole Polytechnique, nizar.touzi@polytechnique.edu}
}
\date{}
\maketitle{}
\begin{abstract}
	This paper extends the results of Ma, Wu, Zhang, Zhang \cite{ma2015} to the context of path-dependent multidimensional forward-backward stochastic differential equations (FBSDE). By path-dependent we mean that the coefficients of the forward-backward SDE at time $ t $ can depend on the whole path of the forward process up to time $ t $. Such a situation appears when solving path-dependent stochastic control problems by means of variational calculus. At the heart of our analysis is the construction of a decoupling random field on the path space. We first prove the existence and the uniqueness of decoupling field on small time interval. Then by introducing the characteristic BSDE, we show that a global decoupling field can be constructed by patching local solutions together as long as the solution of the characteristic BSDE remains bounded. Finally, we provide a stability result for path-dependent forward-backward SDEs.
\end{abstract}



\section{Introduction}
\label{sec:intro}
Forward-backward SDEs appears naturally while solving stochastic control problems by means of variational calculus. Provided that an optimal control exists, the forward component describes the optimal state equation, while the backward component stands for the costate. Motivated by various applications in mathematical finance and their close links with quasi-linear PDEs, the wellposedness of forward-backward SDE (FBSDE) has been extensively studied during the past two decades. A FBSDE is a system of the form
\begin{numcases}
	\mathrm{d}X_{t} = b_{t}(X_t, Y_{t}, Z_{t})\mathrm{d}t + \sigma_{t}(X_t, Y_{t}, Z_{t})\mathrm{d}W_{t} & $ X_{0} = x $, \nonumber \\
	\mathrm{d}Y_{t} = -f_{t}(X_t, Y_{t}, Z_{t})\mathrm{d}t + Z_{t}\mathrm{d}W_{t} & $ Y_{T} = g(X_T) $, \nonumber
\end{numcases}
where $ X $ is called the forward process and $ (Y, Z) $ is called the backward process. The triple $ (X, Y, Z) $ can be multidimensional and the above notations represent in general a system of equations. Besides, the coefficients $ (b,\sigma,f,g) $ of the FBSDE can be random or deterministic. The deterministic coefficient setting is often refered to as the Markovian FBSDE. 

The system is called decoupled if neither $ b $ nor $ \sigma $ depend on $ (Y,Z) $. The decoupled problem reduces to the a Backward SDE as introduced by Pardoux and Peng in \cite{PP}, its wellposedness established under Lipschitz conditions in \cite{PP} has been extended to various situations in the subsequent extensive literature.

In the case of strongly coupled FBSDE, the wellposedness is far more complex and counter-examples under very simple forms can be found: existence of solutions may fail, even locally, and uniqueness may be lost in simple situations, see Example \ref{contreExemple} below. The first wellposedness results are obtained by the so-called four steps scheme in Ma, Protter and Yong \cite{ProtterMY}, under appropriate conditions on the coefficients. The unique global solution is expressed as $Y_t=u(t,X_t)$ and $Z_t=\sigma(t,X_t,u(t,X_t))u_x(t,X_t)$, where the function $u$ is the solution of the quasilinear PDE
\begin{equation}\label{quasiLinearPDE}
\partial_tu+\frac{1}{2}\sigma^2(\cdot,u)u_{xx}+b(\cdot,u,\sigma(\cdot,u)u_{x})u_x+f(\cdot,u,\sigma(\cdot,u)u_x) = 0,~u\big|_{t=T}=g.
\end{equation}
In the non-Markovian FBSDE case, the method of continuation initially introduced by Hu and Peng \cite{Hu1995}, Peng and Wu \cite{PW99}, and later developed by Yong \cite{Yong1997} and recently \cite{Yong2009}, has been widely used in various applications, see e.g. \cite{Yu13, WY2014}. However, the method depends crucially on the monotonicity conditions on the coefficients, which is restrictive in a different way comparing to four steps scheme. Later, using the notion of decoupling field, a general technique has been developed after a series of works of Cvitani{\'c} and Zhang \cite{Cvitanic2005}, Zhang \cite{Zhang2006}, Ma, Wu, Zhang, Zhang \cite{ma2015} and is used to extend the contraction method proposed by Antonelli \cite{antonelli1993}, Pardoux and Tang \cite{PardouxTang} to construct solutions on large intervals by patching together local solutions. 

The decoupling field $u$ can be seen as the non-Markovian substitute to the above quasi-linear PDE. Generally speaking, the decoupling field is a random function such that the solutions of the FBSDE satisfy $:[0,T]\times\mathbb{R}^d\times\Omega$ such that $Y_t=u_t(X_t)$, $t\in[0,T]$.
The key idea is to reduce the FBSDE to the wellposedness problem of the decoupling field. The method was initiated by Ma, Wu \& Zhang \cite{ma2015} in the one-dimensional setting, and further extended by Zhang \cite{Jianfeng02} to the multidimensional case. Fromm and Imkeller in \cite{Fromm2013} redefined the decoupling field using weak derivatives and applied it to general multidimensional FBSDE and defined the notion of maximal interval of a FBSDE.

All of the previous works assume that the possibly random coefficients only depend on the current value of the state $X_t$. Our objective is to allow for an additional possible dependence on the path of the state process, in a progressive way, a situation which arises naturally in path-dependent stochastic control problems which are crucial in various stochastic differential gales as the so-called Principal-Agent problem. 

This paper is largely inspired by the paper \cite{ma2015} and our main contribution is the extension of the existing results to the path-dependent FBSDE defined by path dependent coefficients 
$$
b_{t}(X, Y_{t}, Z_{t}),~\sigma_{t}(X, Y_{t}, Z_{t}),~f_{t}(X, Y_{t}, Z_{t}),
~\mbox{and}~g(X).
$$
By introducing a new metric on the path space (see Assumption \ref{A1}), we first extend the local existence result of \cite{antonelli1993} to path-dependent multidimensional FBSDEs. We then define the characteristic BSDE as in the classical case and by the same spirit of \cite{ma2015, Jianfeng02}, we construct the decoupling field on the path space using the notion of dominating ODE and the comparison principle of quadratic BSDE. Finally we give the stability property of path-dependent FBSDEs.

Note that in the general literature, the FBSDE whose coefficients depend on $ \omega\in\Omega $ are referred to as non-Markovian FBSDE. In order to distinguish our framework with the classical ones and avoid confusion, we call throughout the paper FBSDE whose coefficients depend on $ X_{\cdot\wedge t} $ at time $ t $ path-dependent FBSDE. 

The paper is organized as follows. Section \ref{S2} defines the notion of decoupling field for path-dependent FBSDEs and shows how they lead to the wellposedness of FBSDEs. Section \ref{S3} provides our local existence and uniqueness result for strongly coupled path-dependent FBSDEs. We next show in Section \ref{S4} that such solutions can be made global by analyzing the dynamics of the corresponding decoupling field which leads as in \cite{ma2015} to the wellposedness of some dominating ODE. Finally, Section \ref{S5} contains a stability result for path-dependent multidimensional FBSDEs.

\section{Notations and General Assumptions} \label{S2}
Throughout this paper, we denote $ (\Omega, \mathcal{F}, \mathbb{P}, \mathbb{F}) $ a filtered probability space on which is defined a $ n $ dimensional Brownian motion $ W = (W_{t})_{t\geq0} $. Denote $ \mathbb{F}:=\{\mathcal{F}^{W}_{t}\}_{t\geq0} $, the natural filtration generated by $ W $, augmented by the $ \mathbb{P} $-null sets of $ \mathcal{F} $. For $ t\geq0 $, denote $ \mathbb{H}^2_t(\mathbb{F},\mathbb{R}^n) $ the space of continuous $ \mathbb{F} $-adapted processes on $ [0,t] $ taking values in $ \mathbb{R}^n $ satisfying the integrability condition $\|Y\|_{\mathbb{H}^2_t} := \mathbb{E}\big[ \big(\int_{0}^{t}|Y_s|^2\mathrm{d}s\big)^{\frac{1}{2}} \big]<+\infty$, where $ |\cdot| $ is the Euclidean distance. 

Given $ T>0$, we denote by $\mathcal{C}([0,T],\mathbb{R}^{d})$ the canonical space of continuous paths, that we endow with the norm 
$$ 
\|\mathrm{x}\|^{2}_{2,t} 
:= 
\int_{0}^{t}|\mathrm{x}(s)|^{2}\mathrm{d}s + |\mathrm{x}(t)|^{2},
~\mathrm{x}\in \mathcal{C}([0,T],\mathbb{R}^{d}). 
$$
Let $ \Theta := \mathcal{C}([0,T],\mathbb{R}^{d})\times\mathbb{R}^{n}\times\mathcal{M}_{n}(\mathbb{R}) $, and consider the coefficients
\begin{equation*}
			(b,\sigma,f):[0, T]\times\Omega\times\Theta
			\longrightarrow
			\mathbb{R}^{d}\times\mathcal{M}_{d,n}\times\mathbb{R}^{n}
			\text{, }
			g:\Theta\times\Omega\rightarrow\mathbb{R}^{n}.
		\end{equation*}	
This paper studies the following fully coupled forward-backward stochastic differential equation (abbreviated FBSDE) on $ [0,T] $:
\begin{numcases}{}
\mathrm{d}X_{t} = b_{t}(X, Y_{t}, Z_{t})\mathrm{d}t + \sigma_{t}(X, Y_{t}, Z_{t})\mathrm{d}W_{t} & $ X_{0} = x $, \label{forward} \\
\mathrm{d}Y_{t} = -f_{t}(X, Y_{t}, Z_{t})\mathrm{d}t + Z_{t}\mathrm{d}W_{t} & $ Y_{T} = g(X) $. \label{backward}
\end{numcases}
Throughout the paper, we shall make use of the following standard Lipschitz assumptions.
\begin{assum}\label{A1}
	\begin{enumerate}[label=(\roman*)]
		\item The coefficients $\xi=b,\sigma,f$ are $\mathbb{F} $-progressively measurable for all fixed $(x,y,z)$, and are Lipschitz in the spacial variable: there exists $ K_{0}>0 $ such that  
		\begin{equation}\label{LipschitzCondition}
			|\xi_t(\omega,\theta) - \xi_t(\omega,\theta')|\leq K_{0}\big(\|\mathrm{x}-\mathrm{x}'\|_{2,t}+|y-y'|+|z-z'|\big),
			~\mbox{for all}~ 
			\theta=(\mathrm{x},y,z),\theta' = (\mathrm{x}',y',z'),
		\end{equation}
		uniformly in $ \omega\in\Omega $.
		In particular, denote $ |\nabla_z\sigma|_\infty $ the Lipschitz coefficient of the function $ \sigma $ with respect to $ z\in\mathcal{M}_n(\mathbb{R}) $.
		\item The terminal condition $g$ is jointly measurable, and satisfies the following Lipschitz condition: there exists $ K_{1}>0 $ such that 
		\begin{equation}\label{Lipschitz}
			|g(\omega,\mathrm{x}) - g(\omega\mathrm{x}')|\leq K_{1}\|\mathrm{x}-\mathrm{x}'\|_{2,T},~\mbox{for all}~
			\mathrm{x}, \mathrm{x}'\in\mathcal{C}([0,T],\mathbb{R}^{d}),
			\omega\in\Omega.
		\end{equation}
		\item The coefficients $ \xi^0_t(\omega):=\xi_t(\omega,0) $ for $ \xi=b,f,\sigma,g $ satisfy the integrability condition
		\begin{equation}\label{integrability}
			\mathbb{E}\Big[\Big(\int_{0}^{T}\big(\big|b^0_t\big|+\big|f^0_t\big|\big)\mathrm{d}t\Big)^{2} + \int_{0}^{T}\big|\sigma^0_t\big|^{2}\mathrm{d}t + \big|g^0\big|^2 \Big]<\infty.
		\end{equation}

	\end{enumerate}

\begin{rem}
In the classical literature, Markovian FBSDEs refer to systems whose the coefficients are deterministic and non-Markovian FBSDEs refer to systems with random coefficients, i.e. coefficients depending on the Brownian motion. In both cases, the coefficients of the FBSDE at time $ t $ depend only on the value at time $ t $ of the forward process $ X $. It is worth noting that the aforementioned cases are covered by our setting. 
\end{rem}

\end{assum}

\section{The decoupling field}
A general technique for solving a FBSDE, initiated by Protter, Ma and Yong in \cite{ProtterMY} then further developed by various authors in \cite{MA199759}, \cite{Ma1999}, \cite{PardouxTang}, \cite{Delarue},\cite{Cvitanic2005},\cite{Zhang2006}, \cite{Ma+Yong}, \cite{ma2015}, consists in finding a decoupling function $ u $ such that the $Y-$component of the solution of the FBSDE can be expressed as $Y_t=u(t,X_t)$. In the so-called Markovian case, the function $ u $ is identified with the solution of the quasilinear PDE outlined in the introduction Section \ref{sec:intro}, either in the classical sense or in the viscosity sense. In the case where the coefficients are allowed to be random, the function $ u $ is shown to be the solution of some backward stochastic PDE or is constructed as a random field using localization technique under certain conditions of the coefficients. In this section, we will extend the notion of decoupling field to path-dependent FBSDE. More precisely, we have the following definition of the decoupling field. 
\begin{defn}{\it 
	An $ \mathbb{F} $-progressively measurable random field $ u:[0,T]\times\Omega\times\mathcal{C}([0,T],\mathbb{R}^{d})\rightarrow\mathbb{R}^{n} $ with $ u(T,\omega, \mathrm{x}) = g(\omega, \mathrm{x}) $ is said to be a \textit{decoupling field} of FBSDE \eqref{forward}-\eqref{backward} if there exists a constant $ \delta>0 $ such that, for any $ 0\leq t_{1}<t_{2}\leq T $ with $ t_{2}-t_{1}\leq\delta $ and any $ \mathrm{x}\in \mathbb{H}^2_{t_1}(\mathbb{F},\mathbb{R}^d) $, the FBSDE \eqref{forward}-\eqref{backward} with initial value $ X_{\cdot\wedge t_1} = \mathrm{x} $ and terminal condition $ Y_{t_2} = u(t_{2},\cdot) $ has an unique solution that satisfies
	\begin{equation*}
		Y_{t} = u(t,\omega, X_{\wedge t}) = u(t,\omega,X)\text{,    }t\in[t_1,t_2]\text{, }\mathbb{P}-a.s.
	\end{equation*}
	A decoupling field $ u $ is called regular if it is Lipschitz with respect to $ \mathrm{x} $: there exists $ K>0 $ such that for all $ \mathrm{x},\mathrm{x}'\in\mathcal{C}([0,T],\mathbb{R}^{d}),\forall t\in[0,T] $, we have
	\begin{equation*}
		|u(t,\omega,\mathrm{x})-u(t,\omega,\mathrm{x}')| \leq K\|\mathrm{x}-\mathrm{x}'\|_{2,t},\mathbb{P}-a.s.
	\end{equation*}	
	For notation simplicity, denote $ u_t(X) := u(t,\omega, X) $.}
\end{defn}

Note that the existence of decoupling fields implies the well-posedness of FBSDE over a small time interval. The following result shows the implication of the existence of decoupling fields for the well-posedness of FBSDEs over an arbitrary duration. 
\begin{thm}\label{GlobleExistence}
	Assume that there exists a decoupling field $ u $ for the FBSDE \eqref{forward}-\eqref{backward}. Then, under Assumption \ref{A1},  the FBSDE \eqref{forward}-\eqref{backward} has a unique solution $ (X,Y,Z) $ and 
	\begin{equation*}
	\forall t\in[0,T], Y_{t} = u_t(X).
	\end{equation*}
\end{thm}
The theorem is a direct generalisation of \cite[Theorem 2.3]{ma2015}. For the readers' convenient, we shall detail the proof in Section \ref{S4}.

\section{Main Results} \label{S3}
\subsection{Local wellposedness of the FBSDE}
The local existence of non-Markovian FBSDE is a well-understood problem using the fixed-point approach, see for example in the book of Cvitanic and Zhang \cite{Cvitanic+Zhang}. The following Theorem generalizes the local existence result to path-dependent FBSDEs. 

\begin{thm}\label{thm_contraction}
	Under Assumption \ref{A1}, if $ K_{1}|\nabla_z\sigma|_\infty<1 $, then there exists $ \delta>0 $ such that for all $ T<\delta $, the FBSDE \eqref{forward}-\eqref{backward} has an unique solution $ (X,Y,Z) $. 
\end{thm}
The proof of Theorem \ref{thm_contraction} is reported in Section \ref{S4}. For completeness, we recall the following well-known example which shows that the condition $ K_1|\nabla_z\sigma|_\infty<1 $ is necessary. 

\begin{exmp}[Delarue \cite{Delarue}]\label{contreExemple}
Let $ k $ be a constant and consider the following FBSDE
	\begin{numcases}{}
	\mathrm{d}X_{t} = (k+Z_t)\mathrm{d}W_{t} & $ X_{0} = x $  \nonumber \\
	\mathrm{d}Y_{t} = Z_{t}\mathrm{d}W_{t} & $ Y_{T} = X_T $ . \nonumber
	\end{numcases}
Then, $Y_t = X_T - \int_{t}^{T}Z_s\mathrm{d}s = X_t + k(W_T - W_t).$ The case $ k\neq0 $ leads to the contradiction $ Y_0 - x = kW_T $, while the case $ k=0 $ leads to infinitely many solutions of the FBSDE.
\end{exmp}

\begin{rem}\label{remk1}
	\begin{enumerate}
		\item The length of the interval given by Theorem \ref{thm_contraction} depends on the parameters of the FBSDE, more specifically $ K_0 $, $ K_1 $ and the product $ K_1|\nabla_z\sigma|_\infty $. The larger the Lipschitz constants $ K_0 $ and $ K_1 $, the smaller the interval on which Theorem \ref{thm_contraction} applies.
		
		\item If the forward and backward process are one dimensional and if $ z\to\sigma_t(x,y,z) $ and $ x\to g(x) $ are both invertible, the local existence of non-Markovian FBSDE in the case $ K_1|\nabla_z\sigma|_\infty>1 $ can be proved by using a time inversion technique, see Theorem 6.2 in \cite{ma2015}.
		
		\item The local existence and uniqueness of the solution of FBSDE given in Theorem \ref{thm_contraction} provides a way to construct the decoupling field on a small time interval. More precisely, consider the FBSDE \eqref{forward}-\eqref{backward} on $ [0,T] $ satisfying the assumptions in Theorem \ref{thm_contraction}. For any $ t\in[0,T] $, for any $ \mathrm{x}\in\mathcal{C}([0,t],\mathbb{R}^{d}) $, the same FBSDE on $ [t,T] $ with initial condition $ X_{t_1}=\mathrm{x} $ and terminal condition $ Y_T=g(X) $ still has a unique solution. Let $ u_t(x) := Y_{t} $. One can check easily that $ u $ is the unique associated decoupling field.
	\end{enumerate}
\end{rem}

\subsection{Existence of Regular Decoupling Field}
We next follow the same line of argument as in \cite{ma2015} in order to extend the local existence result to larger time interval. The two important ingredients for local existence in Theorem \ref{thm_contraction} are the Lipschitz constant of the terminal condition smaller than $ |\nabla_z\sigma|_\infty^{-1}$, and the time interval shorther than the constant $ \delta_0 $ determined by the coefficients of the FBSDE.

The strategy of constructing a decoupling field on a larger time interval is the following: first we construct a decoupling field $ u $ on $ [T-\delta_0, T] $ by solving locally the FBSDE and we estimate the upper bound of the gradient of the decoupling field $ u $ with respect to the forward process, which is denoted $ K_{T-\delta_0} $ and will be used as the Lipschitz constant of the terminal condition when we then try to construct a decoupling field on $ [T-\delta_0-\delta_1, T-\delta_0] $ for some $ \delta_1 $. If the condition $ K_{T-\delta_0}|\nabla_z\sigma|_\infty<1 $ is still satisfied at $ T-\delta_0 $, we can proceed the same way and extend the local existence and uniqueness step by step until the whole interval is covered.

Notice that obtaining a bound on gradient of the decoupling field $ u $ with respect to the forward process is equivalent to find an upper bound of the corresponding variational FBSDE, which will be defined in the following Section \ref{case-1}. The technique consists in using the comparison principle of quadratic BSDE given by Kobylanski in \cite{Kobylanski} to find a dominating ODE, the solution of which, if exists on $ [0,T] $, dominates the variational FBSDE. This method is used in \cite{ma2015} in the context of one-dimensional non-Markovian FBSDE then generalized in \cite{Jianfeng02} to the case where the backward process is multidimensional. In this subsection, we generalize the existing results to different types of path-dependent FBSDEs. We shall begin by showing the wellposedness of decoupled FBSDE in Section \ref{case0} as a direct consequence of the existence of dominating ODE. In Section \ref{case1}, we shall study the case where $ b_t(X,Y_t,Z_t) = b_t(X,Y_t) $ and $ \sigma_t(X,Y_t,Z_t) = \sigma_t(X) $. These are the only cases where global existence can be guaranteed for arbitrary duration under some conditions. The case where $ \nabla_z\sigma=0 $ is discussed in Section \ref{case3} and the general case at Section \ref{case4}. In both cases, the corresponding dominating ODE is a Riccati equation and we introduce the notion of maximal interval as in \cite{Fromm2013}.

\subsubsection{Variational FBSDE, characteristic BSDE and dominating ODE} \label{case-1}
\begin{note}
	\begin{enumerate}[label=(\roman*)]
		\item For $ i\in \llbracket 1,d \rrbracket $, denote $ X_{i} $ the i-th component of $ X $. The corresponding 1-dimensional forward equation is
		\begin{equation*}
		\mathrm{d}X_{i,t} = b_{i,t}(X, Y_{t}, Z_{t})\mathrm{d}t + \sigma^\intercal_{i,t}(X, Y_{t}, Z_{t})\mathrm{d}W_{t},
		\end{equation*}
		where $ b_{i,t} $ and $ \sigma_{i,t} $ is the $ i $-th component of the vector $ b_t $ and the $ i $-th row of the matrix $ \sigma_t $, respectively.
		
		\item For $ i\in \llbracket 1,n \rrbracket $, denote $ Y_{i} $ (respectively $ Z_{i} $) the i-th component of $ Y $ (respectively the $ i $-th row of $ Z $). The corresponding component-wise backward equation is
		\begin{equation*}
			\mathrm{d}Y_{i,t} = -f_{i,t}(X, Y_{t}, Z_{t})\mathrm{d}t + Z^\intercal_{i,t}\mathrm{d}W_{t},
		\end{equation*}
		where $ f_{i,t} $ is the $ i $-th component of the vector $ f_t $.
		
		\item For $ \theta,\theta'\in\Theta $, denote $ \Delta\theta = (\Delta \mathrm{x},\Delta y,\Delta z) := \theta - \theta' $. For $ \xi=b,\sigma,f $, denote
		\begin{equation}\label{grandient}
			\xi_{x,t}(\theta,\theta') := \frac{\xi(\mathrm{x},y,z) - \xi(\mathrm{x}',y',z')}{\|\Delta \mathrm{x}\|_{2,t}}; 
		\end{equation}
		\begin{equation*}
			\xi_{y,t}(\theta,\theta') := \Bigg(\frac{\xi(\mathrm{x}',y_{1}',\cdots,y_{k-1}',y_{k},\cdots,z) - \xi(\mathrm{x}',y_{1}',\cdots,y_{k}',y_{k+1},\cdots,z)}{y_{k}-y_{k}'}\Bigg)_{k \in \llbracket 1,n \rrbracket};
		\end{equation*}
		\begin{equation*}
			\xi_{z,t}(\theta_{1},\theta_{2}) := \Bigg(\frac{\xi(\mathrm{x}',y',z_{1,1}',\cdots z_{k,l},\cdots,z_{n,n}) - \xi(\mathrm{x}',y',z_{1,1}'\cdots,z_{k,l}',\cdots,z_{1,n,n})}{z_{k,l}-z_{k,l}'}\Bigg)_{k,l \in \llbracket 1,n \rrbracket};
		\end{equation*}
		where $ y_{k} $ is the k-th component of the vector $ y $ and $ z_{k,l} $ is the value of the component at the position $ (k,l) $ of the matrix $ z $. Here and in the sequel, for any Lipschitz continuous function $ \xi(x) $, when $ x = x' $, we will always take the convention that $\frac{\xi(x)-\xi(x)}{x-x}:=\inflimit_{\tilde{x}\to x}\frac{\xi(\tilde{x})-\xi(x)}{\tilde{x}-x}$.
		
		Under the above notation, for $ \xi=b,\sigma,f $, we have 
		\begin{equation*}
			\xi_t(\theta) - \xi_t(\theta') = \xi_{x,t}\|\Delta\mathrm{x}\|_{2,t} + \xi_{y,t}\Delta y_t + \Tr[\xi_{z,t}\Delta z_t].
		\end{equation*}
		
		For notation simplicity, we shall omit the index $ t $ in the following.
	\end{enumerate}	
\end{note}

Let $ X_{\cdot\wedge t_{0}},\tilde{X}_{\cdot\wedge t_{0}} $ two processes on $ \mathbb{H}^2_{t_0}(\mathbb{F},\mathbb{R}^d) $. By Theorem \ref{thm_contraction}, the FBSDE \eqref{forward}-\eqref{backward} has an unique solution $ (X,Y,Z)_{[t_{0}, T]} $ \big(respectively $ (\tilde{X},\tilde{Y}, \tilde{Z})_{[t_{0}, T]} $\big) on $ [t_{0}, T] $ given the initial condition $ X_{\cdot\wedge t_{0}} $ \big(respectively $ \tilde{X}_{\cdot\wedge t_{0}} $\big) if $ T-t_{0} $ is small enough. Denote $ X $ \big(respectively $ \tilde{X} $\big) the concatenation of $ X_{\cdot\wedge t_{0}} $ and $ X_{[t_{0}, T]} $ \big(respectively $ \tilde{X}_{\cdot\wedge t_{0}} $ and $ \tilde{X}_{[t_{0}, T]} $\big).

Let $ \mathcal{X} := X - \tilde{X} $, $ \mathcal{Y}:=Y - \tilde{Y} $ and $ \mathcal{Z}:=Z - \tilde{Z} $. We can check easily that $ (\mathcal{X}, \mathcal{Y}, \mathcal{Z}) $ is a solution of the following variational FBSDE associated to the original FBSDE \eqref{forward}-\eqref{backward} on the time interval $ [t_{0}, T] $:
\begin{numcases}{}
\mathrm{d}\mathcal{X}_{t} = (b_{x}\mathcal{D}_{t}+b_{y}\mathcal{Y}_{t}+\Tr(b_{z}\mathcal{Z}_{t}))\mathrm{d}t + (\sigma_{x}\mathcal{D}_{t}+\sigma_{y}\mathcal{Y}_{t}+\Tr(\sigma_{z}\mathcal{Z}_{t}))\mathrm{d}W_{t} & $ \mathcal{X}_{t_{0}} = x_{0} $, \\
\mathrm{d}\mathcal{Y}_{t} = (f_{x}\mathcal{D}_{t}+f_{y}\mathcal{Y}_{t}+\Tr(f_{z}\mathcal{Z}_{t}))\mathrm{d}t - \mathcal{Z}_{t}\mathrm{d}W_{t} & $ \mathcal{Y}_{T} = \Delta g $,
\end{numcases}
where $ \mathcal{D}_{t}:= \|\mathcal{X}\|_{2,t} $ and $ \Delta g := g(X) - g(\tilde{X}) $. Define
\begin{equation*}
	\mathcal{H}_t := \frac{\mathcal{Y}^\intercal_{t}\mathcal{Y}_{t}}{\mathcal{D}^{2}_{t}},
	\alpha_{t} := \frac{\mathcal{Z}_{t}}{\mathcal{D}_{t}}\text{,  }\beta_{t} := \frac{\mathcal{X}_{t}}{\mathcal{D}_{t}}\text{,  }P_{t} := \frac{\mathcal{Y}_{t}}{|\mathcal{Y}_{t}|},
\end{equation*}
\begin{equation*}
	\mathrm{d}\tilde{W}_{t} := \mathrm{d}W_{t} - 2(\sigma_{x} + \sigma_{y}P_{t}\sqrt{\mathcal{H}_{t}} + \Tr(\sigma_{z}\alpha_{t}))^\intercal\beta_{t}\mathrm{d}t.
\end{equation*}

Then it follows from Itô's formula that
\begin{equation}\label{Hequation}
	\mathrm{d}\mathcal{H}_t = -\mathcal{F}_t(\mathcal{H}_t)\mathrm{d}t + \mathcal{N}_t\mathrm{d}\tilde{W}_t\text{ with }\mathcal{F}_t(h) = A_th^2 + B_th^{3/2} + C_th + D_th^{1/2} + F_t,
\end{equation}
\begin{equation*}
	A_t = \Tr(\sigma_yP_tP_t^\intercal\sigma_y^\intercal) - 8\beta^\intercal_tP_t^\intercal\sigma_y^\intercal\sigma_yP_t\beta_t,
\end{equation*}
\begin{equation*}
	B_t = 2\beta^\intercal b_yP_t + 2\Tr\big(\sigma_yP_t\big(\sigma_x + \Tr(\sigma_z\alpha_t)\big)^\intercal\big) - 16
	\beta^\intercal P_t^\intercal\sigma_y^\intercal\big(\sigma_x + \Tr(\sigma_z\alpha_t)\big)\beta_t,
\end{equation*}
\begin{equation*}
	C_t = 2P_t^\intercal f_yP_t + |\beta_t|^2 +2\beta_t^\intercal\big(b_x + \Tr(b_z\alpha_t)\big) + \Tr(\sigma_x\sigma^\intercal_x) - 8\beta_t^\intercal\sigma_x^\intercal\sigma_x\beta_t,
\end{equation*}
\begin{equation*}
	D_t = 2P_tf_x + 2P^\intercal_t\Tr(f_z\alpha_t)\text{, }F_t = -\Tr(\alpha_t\alpha^\intercal_t),
\end{equation*}
\begin{equation*}
	\mathcal{N}_t = 2\mathcal{H}^{1/2}P^\intercal_t\alpha_t - 2\mathcal{H}_{t}\beta^\intercal_{t}(\sigma_{x} + \sigma_{y}P_{t}\mathcal{H}^{1/2}_{t} + \Tr(\sigma_{z}\alpha_{t})).
\end{equation*}
We call the equation \eqref{Hequation} the characteristic BSDE of the FBSDE.

\begin{defn}
	Let $ G:[0,T]\times\mathbb{R}\to\mathbb{R} $ be a continuously differentiable function. The ODE
	\begin{equation*}
		\dot{y}_t = -G(t, y_t)
	\end{equation*}
	is called a dominating ODE of the FBSDE \eqref{forward}-\eqref{backward} if $ G $ satisfies the following conditions:
	\begin{enumerate}[label=(\roman*)]
		\item for all $ t\in[0, T] $, we have $\mathcal{F}_t(\cdot)\leq G(t,\cdot) $, $ \mathbb{P} $-almost surely;
		\item for all $ M>0 $ there exists $ l, \hat{l} \in L^{1}([0,T],\mathbb{R}) $ such that 
		\begin{equation*}
			|G(t,h)| \leq l(t)\text{  and  }\left|\frac{\partial G}{\partial h}(t,h)\right| \leq \hat{l}(t),
			~\mbox{for all}~t\in[0,T], h\in[-M,M].
		\end{equation*}
	\end{enumerate}
\end{defn}

The following proposition is a direct adaptation from \cite[Theorem 2.3]{ma2015}.
\begin{prop} \label{prop_ODE}
	Assume that Assumption \ref{A1} holds true and that there exists a continuously differentiable function $ G:[0,T]\times\mathbb{R}\to\mathbb{R} $ such that $ \dot{y}_t = -G(t, y_t) $ is a dominating ODE of the FBSDE \eqref{forward}-\eqref{backward}. If the ODE has a bounded solution on $ [0,T] $, then the FBSDE has a unique regular decoupling field on $ [0,T] $ and therefore, it is well-posed.
\end{prop}
\begin{proof}[\textbf{Proof of Proposition \ref{prop_ODE}}]
	Let $ T>0 $. Let $ y $ be the solution of the dominating ODE:
	\begin{equation}\label{eq: dominating ODE}
	\dot{y}_t = -G(t,y_t)\text{,   }y_T = K^2_1.
	\end{equation}
Denote $ K^2_{max}:=\max_{t\in[0,T]}y_t $, the upper bound of $ y $. Firstly, by the comparison principle, we have $ \mathcal{H}_t\leq y_t\leq K^2_{max} $ for $ t\in[T-\epsilon_0, T] $ where $ \epsilon_0 $ is a constant depending on $ K_0 $, $ K_1 $, $ n $ and $ d $ given by Theorem \ref{thm_contraction}. This can be reformulated as follow using the decoupling field: for all $ t\in[T-\epsilon,T] $, for all given initial condition for the forward process $ X_{T-\epsilon_0} = x\in\mathbb{H}^2_{T-\epsilon_0}(\mathbb{F},\mathbb{R}^d) $, $ \mathbb{P} $-almost surely,
	\begin{equation}
	|u(t,X) - u(t,X')|^2\leq y_t\|X-X'\|^2_{2,t} \leq K^2_{max}\|X-X'\|^2_{2,t}.
	\end{equation}	
	To finish the proof, we only need to repeat the same procedure at $ T-\epsilon_0 $ and so on. Again by Theorem \ref{thm_contraction}, we can find $ \epsilon_1 $ such that the FBSDE has a unique solution on $ [T-(\epsilon_1 + \epsilon_0), T-\epsilon_0] $ and for all $ t\in[T-(\epsilon_1 + \epsilon_0), T-\epsilon_0] $, for all given initial condition for the forward process $ X_{T-\epsilon_0-\epsilon
		_1} = x\in\mathbb{H}^2_{T-\epsilon_0-\epsilon_1}(\mathbb{F},\mathbb{R}^d) $, 
	\begin{equation}
	|u(t,X) - u(t,X')|^2\leq y_t\|X-X'\|^2_{2,t} \leq K^2_{max}\|X-X'\|^2_{2,t}.
	\end{equation}
	Notice since $ K_{max} $ dominates the Lipschitz constants of the decoupling field $ u(t,\cdot) $ for all $ t\in[0,T] $, we can choose each $ \epsilon_i\geq\bar{\epsilon} $ where $ \bar{\epsilon} $ is a constant given by Theorem \ref{thm_contraction} when applied to a FBSDE with Lipschitz constant $ K_0 $ and $ K_{max} $. Therefore, by iterating at most $ T/\bar{\epsilon} $, we construct a decoupling filed for the FBSDE \eqref{decoupledForwardeq}-\eqref{decoupledBackwardeq} on $ [0,T] $ and by Theorem \ref{GlobleExistence}, the FBSDE has an unique solution.
\end{proof}

\subsubsection{Decoupled Path-dependent FBSDE} \label{case0}
Consider the following decoupled path-dependent FBSDE:
\begin{numcases}
	\mathrm{d}X_{t} = b_{t}(X)\mathrm{d}t + \sigma_{t}(X)\mathrm{d}W_{t} & $ X_{0} = x $  \label{decoupledForwardeq}\\
	\mathrm{d}Y_{t} = -f_{t}(X, Y_{t}, Z_{t})\mathrm{d}t + Z_{t}\mathrm{d}W_{t} & $ Y_{T} = g(X) $ \label{decoupledBackwardeq}. 
\end{numcases}
The decoupled FBSDEs are always wellposed under standard Lipschitz assumptions because one can always solve independently the forward process then inject the solution into the backward equation and solve it as a standard BSDE. Another way to prove the wellposedness is to show the existence of a unique decoupling field of the FBSDE, which, in the decoupled case, is guaranteed by Proposition \ref{prop_ODE}. More precisely, in this case we have $ A_t=B_t=0 $ and the characteristic BSDE \eqref{Hequation} becomes
\begin{equation*}
	\mathrm{d}\mathcal{H}_t = -\big( C_t\mathcal{H}_t + D_t\mathcal{H}^{1/2}_t + F_t \big)\mathrm{d}t + \mathcal{N}_t\mathrm{d}\tilde{W}_t.
\end{equation*} 
One can find a linear dominating ODE with bounded solution on $ [0,T] $.

\begin{prop}\label{thm_decoupledFBSDE}
	Consider the decoupled path-dependent FBSDE  \eqref{decoupledForwardeq}-\eqref{decoupledBackwardeq}. Under Assumptions \ref{A1}, for all $ T>0 $, the equation \eqref{decoupledForwardeq}-\eqref{decoupledBackwardeq} has an unique solution on $ [0,T] $.
\end{prop}

\subsubsection{The case $ b = b_t(x,y) $ and $ \sigma = \sigma_t(x) $} \label{case1}

In this case, $ A_t=0 $ and the characteristic BSDE \eqref{Hequation} becomes
\begin{equation}\label{eq: characteristic BSDE 2}
	\mathrm{d}\mathcal{H}_t = -\big( B_t\mathcal{H}^{3/2}_t + C_t\mathcal{H}_t + D_t\mathcal{H}^{1/2}_t + F_t \big)\mathrm{d}t + \mathcal{N}_t\mathrm{d}\tilde{W}_t.
\end{equation} 
\begin{thm}\label{thm: 3.5}
	Let $ T>0 $, $ b = b_t(x,y) $, $ \sigma = \sigma_t(x) $. Let Assumption \ref{A1} hold true, and
	\begin{align}
	&(b_t(\theta) - b_t(\theta'))^\intercal b_y\Delta y 
	   -(f_t(\theta) - f_t(\theta'))^\intercal b_y^\intercal\Delta\mathrm{x}_t  \nonumber \\
	&\qquad\qquad\qquad\qquad + \Tr((\sigma_t(\theta) - \sigma_t(\theta'))^\intercal b_y\Delta z) \geq (g(\mathrm{x}_1) - g(\mathrm{x}_2))^\intercal b_y^\intercal\mathrm{x}_T, \label{conditionMonotonicity} 
	\end{align}
	for all $t\in[0, T]$, $ \theta=(\mathrm{x},y,z),\theta'=(\mathrm{x},y,z)\in\mathcal{C}([0,T],\mathbb{R}^d)\times\mathbb{R}^n\times\mathcal{M}_{n}(\mathbb{R}) $, and all
	\begin{equation*}
	b_y\in
		\mathcal{B}_y := \left\{ b_y(t,\mathrm{x},y_1,\mathrm{x},y_2)\in\mathcal{M}_{d,n}(\mathbb{R})\text{  for  }(t,\mathrm{x},y_1,y_2)\in[0,T]\times\mathcal{C}([0,T],\mathbb{R}^{d})\times(\mathbb{R}^n)^2, y_1\neq y_2 \right\}.
	\end{equation*}
	Then the FBSDE has a unique solution on $ [0, T] $.
\end{thm}

\begin{proof}[\textbf{Proof of Theorem \ref{thm: 3.5}}]
	Using the definition of $ \beta $ and $ P $, by Itô's formula, we have
	$$
	\mathcal{H}_t(\beta^\intercal_tb_yP_t\sqrt{\mathcal{H}_t}) 
	= \mathcal{H}_t\frac{\mathcal{X}^\intercal_tb_y\mathcal{Y}_t}{\mathcal{D}^2_t} 
	= \frac{\mathcal{H}_t}{\mathcal{D}_t^2}\mathbb{E}_t\left[\Delta X^\intercal_Tb_y\Delta g - \int_{t}^{T}(\Delta b^\intercal_sb_y\mathcal{Y}_s - \mathcal{X}^\intercal_sb_y\Delta f_s + \Tr(\Delta\sigma^\intercal_sb_y\mathcal{Z}))\mathrm{d}s\right],
	$$
	which is nonpositive by \eqref{conditionMonotonicity}. Since $ B_t $ is non-positive, we may find two constants $ c $ and $ d $ such that $\dot{y}= -cy-d$
	is a dominating ODE for the BSDE \eqref{eq: characteristic BSDE 2}. Since the above ODE has a bounded solution on $ [0,T] $ for all $ T>0 $, the FBSDE has a unique solution by Proposition \ref{prop_ODE}.
\end{proof}

\begin{rem}
	In the one-dimensional case, if $ b $ is increasing in $ y $, we can have the following sufficient condition for \eqref{conditionMonotonicity}, which is easier to verify.
	\begin{numcases}{}
		(b_t(\theta_{1}) - b_t(\theta_{2}))\Delta y - (f_t(\theta_{1}) - f_t(\theta_{2}))\Delta \mathrm{x}_t + (\sigma_t(\theta_{1}) - \sigma_t(\theta_2))\Delta z \geq 0; \nonumber \\
		(g(\mathrm{x}_1) - g(\mathrm{x}_2))\Delta \mathrm{x}_T \leq 0. \nonumber
	\end{numcases}
	Similarly, if $ b $ is decreasing in $ y $, we have
	\begin{numcases}{}
		(b_t(\theta_{1}) - b_t(\theta_{2}))\Delta y - (f_t(\theta_{1}) - f_t(\theta_{2}))\Delta \mathrm{x}_t + (\sigma_t(\theta_{1}) - \sigma_t(\theta_2))\Delta z \leq 0; \nonumber \\
		(g(\mathrm{x}_1) - g(\mathrm{x}_2))\Delta \mathrm{x}_T \geq 0. \nonumber
	\end{numcases}
	This condition shares the same spirit as the monotonicity condition in the continuation method for solving one-dimensional Markovian framework FBSDE introduce by Hu and Peng in \cite{Hu1995}. More details on the continuation methods for solving Markovian FBSDE can be found for example in \cite[Section 11.4]{Cvitanic+Zhang}.
\end{rem}

\subsubsection{The case $ \sigma = \sigma_t(x,y) $} \label{case3}
We recall the following example which shows that Assumption \ref{A1} is not enough in this case for global wellposedness.

\begin{exmp}[Fromm \& Imkeller \cite{Fromm2013}]
	Consider the following fully coupled FBSDE:
	$$
	\mathrm{d}X_t = Y_t\mathrm{d}t ,~X_0 = x,
	~\mbox{and}~
	\mathrm{d}Y_t = Z_t\mathrm{d}W_t,~Y_T = X_T .
	$$
We notice that the condition in Theorem \ref{thm_contraction} is satisfied in this case. Clearly for $ T<1 $ the problem has a unique decoupling field $
		u(t,x) = \frac{x}{1 - (T-t)}$, and we have $X_t = x\frac{1-(T-t)}{1-T}\text{, }Y_t = \frac{x}{1-T}\text{ and }Z_t = 0.$
	Notice that when $ x\neq0 $, $ u $ tend to infinity in the neighbourhood of $ 0 $ when $ T\to1 $, thus there is no decoupling field on $ [0,1] $ for this FBSDE.
\end{exmp}

Now let's consider the following path-dependent FBSDE:
\begin{numcases}
\text{ }\mathrm{d}X_{t} = b_{t}(X, Y_{t}, Z_{t})\mathrm{d}t + \sigma_{t}(X, Y_{t})\mathrm{d}W_{t} & $ X_{0} = x $ \label{coupledForwardeq_2} \\
\mathrm{d}Y_{t} = -f_{t}(X, Y_{t}, Z_{t})\mathrm{d}t + Z_{t}\mathrm{d}W_{t} & $ Y_{T} = g(X) $. \label{coupledBackwardeq_2}
\end{numcases}

The condition in Theorem \ref{thm_contraction} $ K_1 |\nabla_z\sigma|_\infty<1 $ is automatically satisfied. Therefore, there exists $ \epsilon>0 $ such that the FBSDE \eqref{coupledForwardeq_2}-\eqref{coupledBackwardeq_2} has a unique regular decoupling field $ u $ on $ [T-\epsilon, T] $ with the terminal condition $ u_T(X) = g(X) $. Denote $ \mathcal{H} $ the solution of the associated characteristic BSDE \eqref{Hequation}. Note that $ \esssup\mathcal{H}_t $ is a Lipschitz constant of $ u_t $ with respect to the path space variable. Therefore, as long as the solution of the characteristic BSDE $ \mathcal{H} $ is bounded on $ [T-\epsilon,T] $, we can re-apply the local existence result at $ T-\epsilon $ with terminal condition $ Y_{T-\epsilon} = u_{T-\epsilon}(X) $ and so on. Notice that the length of the time interval $ \epsilon $ given by Theorem \ref{thm_contraction} will decrease when the Lipschitz constant of the terminal condition of the backward process increases. The Lipschitz constant that we get at time $ T-\epsilon $ is $ \esssup\mathcal{H}_{T-\epsilon} $, which is always bigger than $ K_1 $. It means that the length of step at which we iterate the procedure decreases. In order to find the maximal time interval on which we can construct a solution by the above procedure, one way is to find a dominating ODE and find the time of the explosion $ T_{max} $ of the ODE. By Proposition $ \ref{prop_ODE} $, for any $ T<T_{max} $, the FBSDE has a unique solution. One possible dominating ODE is
\begin{equation*}
	\dot{y}_t = \left(|A_t|_\infty+\frac{|B_t|_\infty}{2}\right)  y^2_t + \left(|C_t|_\infty+\frac{|B_t|_\infty}{2}+\frac{|D_t|_\infty}{2}\right) y_t + \left(|F_t|_\infty+\frac{|D_t|_\infty}{2}\right),
\end{equation*}
where $ |\cdot|_\infty $ is the essential supremum and the coefficients are given in \eqref{Hequation}. The result is summarized in the following Theorem.

\begin{thm}
	Assume that Assumption \ref{A1} holds true. Then there exists a dominating Riccati ODE with terminal condition $ y_T = K_1 $. In addition, there exists $ T_{max}>0 $ depending only on the dimension and the Lipschitz coefficients of the FBSDE such that the dominating ODE has a bounded solution on $ [0,T] $ for all $ T<T_{\max} $ and hence, the FBSDE \eqref{coupledForwardeq_2}-\eqref{coupledBackwardeq_2} has a unique solution on $ [0,T] $.
\end{thm}

\subsubsection{General Case} \label{case4}
In the general case where $ |\nabla_z\sigma|_\infty\neq 0 $, in order to have the existence on small time interval, we need to have the condition $ K_1|\nabla_z\sigma|_\infty<1 $. To use the same technique to extend the existence result on larger interval, we need to maintain the very same condition, i.e. $ |\nabla_\mathrm{x}u(t,\cdot)|_\infty|\nabla_z\sigma|_\infty<1 $, where $ |\nabla_\mathrm{x}u(t,\cdot)|_\infty $ is the essential supremum of all the directional derivatives of the decoupling field $ u $ with respect to the path space variable $ \mathrm{x}\in\mathcal{C}([0,t],\mathbb{R}^d) $ at time $ t $ as defined in \eqref{grandient}. We introduce now the notion of maximal interval as in \cite{Fromm2013}.

\begin{defn}
	The maximal interval $ I_{\max} $ on $ [0,T] $ for the FBSDE $ (b,\sigma,f,g) $ is defined as the union of all intervals of form $ [t,T] $ on which the FBSDE $ (b,\sigma,f,g) $ has a decoupling field $ u $ such that $ |\nabla_\mathrm{x}u(s,\cdot)|_\infty|\nabla_z\sigma|_\infty<1 $ for all $ s\in[t,T] $.
\end{defn}
\begin{rem}
	Notice that the maximal interval for a FBSDE given $ T $ may very well be open to the left. In this case we say a decoupling field is regular on $ I_{\max} $ if $ u $ restricted to $ [s,T] $ is a regular decoupling field for all $ s\in I_{\max} $.
\end{rem}
\begin{prop}{\rm \cite[Theorem 2]{Fromm2013}}\label{Prop4.9}
	Under Assumption \ref{A1}, if $ K_1|\nabla_z\sigma|_\infty<1 $, let $ I_{\max} $ be the maximal interval associated to the FBSDE $ (b,\sigma,f,g) $, then there exists an unique regular decoupling field $ u $ satisfying $|\nabla_\mathrm{x}u(t,\cdot)|_\infty|\nabla_z\sigma|_\infty<1$.
\end{prop}

\begin{proof}
For any $t\in I_{\max}$, by definition of $I_{\max}$ and Theorem \ref{GlobleExistence}, there exists a unique decoupling field $u^{t}$ on $[t,T]$. For $t_1, t_2 \in I_{\max}$, denote $u^{t_1}, u^{t_2}$ the respective decoupling field on $[t_1,T]$ and $[t_2,T]$ for FBSDE $ (b,\sigma,f,g) $. By the same arguments as the proof of Theorem \ref{GlobleExistence}, one can show that $u^{t_1}$ and $u^{t_2}$ coincides on $[\max(t_1,t_2),T]$ and therefore $u(t,\cdot) := u^{t}(t,\cdot)$ for all $t\in I_{\max}$ is a decoupling field for FBSDE $ (b,\sigma,f,g) $.
\end{proof}

\begin{prop}\label{prop: 3.8}
	Under Assumption \ref{A1} and assume that $ K_1|\nabla_z\sigma|_\infty<1 $, if the maximal interval associated to the FBSDE $ (b,\sigma,f,g) $ is open on the left, i.e. $ I_{\max} = (t_{\min},T] $, then necessarily,
	\begin{equation}\label{leftLimit}
	\lim_{t\downarrow t_{\min}}|\nabla_\mathrm{x}u(t,\cdot)|_\infty|\nabla_z\sigma|_\infty = 1.
	\end{equation}
\end{prop}

\begin{proof}
The same argument as in the Markovian case of \cite{Fromm2013} applies here, we report it for completeness.
	Assume that there exist a sequence of $ (t_n)_{n\geq0}\downarrow t_{\min} $ such that
	\begin{equation*}
	\lim_{t_n\downarrow t_{\min}}|\nabla_\mathrm{x}u(t_n,\cdot)|_\infty|\nabla_z\sigma|_\infty < 1.
	\end{equation*}
	According to Remark \ref{remk1}, one can construct a small time interval $ \epsilon $ depending only on the Lipschitz coefficient of the FBSDE $ K_0 $, $ \limsup_{t_n\downarrow t_{\min}}|\nabla_\mathrm{x}u(t_n,\cdot)|_\infty $ and $|\nabla_z\sigma|_\infty $ such that for $ n $ large enough, we can construct a decoupling field for the FBSDE on the interval $ [t_n-\epsilon, t_n] $. Since $ \epsilon $ is independent of $ n $, one can choose a $ n $ such that $ t_n-\epsilon<t_{\min} $, contradicting the definition of maximal interval.
\end{proof}

\section{Stability of path-dependent multidimensional FBSDE}\label{S5}
Let $ \mathcal{L} $ be the set of all $ \mathbb{F} $-adapted processes $ (Y,Z) $ with
$\|(Y,Z)\|_2 := \sup_{t\in[0,T]}\big\{\mathbb{E}\big[|Y_{t}|^{2} + \int_{t}^{T}|Z_{s}|^{2}\mathrm{d}s\big]\big\}<+\infty$. Consider the path-dependent FBSDE \eqref{forward}-\eqref{backward}, and denote
\begin{equation*}
	I^2_0 := \mathbb{E}\Big[ \Big(\int_{0}^{T}\left|f^0_t\right| + \left|b^0_t\right|\mathrm{d}t\Big)^2 + \int_{0}^{T}\left|\sigma^0_t\right|^2\mathrm{d}t \Big].
\end{equation*}

The following lemma generalizes the existing result on the a priori estimate on FBSDE. The techniques are similar with an additional difficulty that the coefficients of the FBSDE can depend on the whole path of the forward process $ X $.

\begin{lem}[A Priori Estimate for FBSDE on Small Time Interval]\label{estimateFBSDE}
	Assume that all the hypotheses in the Assumption \ref{A1} are satisfied. Let $ T $ be a small time horizon on which Theorem \ref{thm_contraction} applies. If $ (X,Y,Z)\in\mathbb{H}_T^2\times\mathcal{L} $ are solution of the FBSDE \eqref{forward}-\eqref{backward} on $ [0,T] $, then there exists a constant $ C $ such that
	\begin{equation*}
	\sup_{t\in[0,T]}\Big\{\mathbb{E}\Big[\|X\|^2_{2,t} + |Y_{t}|^{2} + \int_{t}^{T}|Z_{s}|^{2}\mathrm{d}s\Big]\Big\} \leq C\big(|x|^2 + \mathbb{E}\left[ |g^0|^2 \right] + I^2_0\big).
	\end{equation*}
\end{lem}
\begin{proof}[\textbf{Proof of Lemma \ref{estimateFBSDE}}]
	Let $ (y, z) $ two progressively measurable processes. Let $ (X,Y,Z) $ be the unique solution of the following decoupled FBSDE on $ [0,T] $:
	\begin{numcases}
	\text{ }\mathrm{d}X_{t} = b_{t}(X, y_{t}, z_{t})\mathrm{d}t + \sigma_{t}(X, y_{t}, z_{t})\mathrm{d}W_{t} & $ X_{0} = x $  \label{eq001} \\
	\mathrm{d}Y_{t} = -f_{t}(X, y_{t}, z_{t})\mathrm{d}t + Z_{t}\mathrm{d}W_{t} & $ Y_{T} = g(X_{\cdot\wedge T}) $. \label{eq002}
	\end{numcases}
	We have shown that the mapping $ (y,z)\mapsto(Y,Z) $ is a contraction in the proof of Theorem \ref{thm_contraction} in the space $ (\mathcal{L}, \|\cdot\|_2) $. Denote $ (X_0,Y_0,Z_0) $ the solution of the FBSDE \eqref{eq001}-\eqref{eq002} with $ (y,z) = (0,0) $.
	We have
	\begin{equation*}
	\|(Y-Y_0,Z-Z_0)\|_2 \leq C\|(Y,Z)\|_2, 
	\end{equation*}
	where $ C<1 $. By the triangle inequality, we get
	\begin{equation*}
	\|(Y,Z)\|_2 
	\leq \|(Y-Y_0,Z-Z_0)\|_2 + \|(Y_0,Z_0)\|_2\leq C\|(Y,Z)\|_2 + \|(Y_0,Z_0)\|_2,
	\end{equation*}
	and therefore, together with standard estimates on SDEs and BSDEs (see e.g. \cite[Chapter 9]{Cvitanic+Zhang}), we have
	\begin{align*}
	\|(Y,Z)\|_2 
	&\leq \frac{1}{1-C}\|(Y_0,Z_0)\|_2 \\
	&\leq C\mathbb{E}\left[ |g(X_0)|^2 + \left( \int_{0}^{T}|f(t,X_0(t),0,0)|\mathrm{d}t \right)^2 \right] \\
	&\leq C\mathbb{E}\left[ K^2_1\|X_0\|^2_{2,T} + K^2_0\int_{0}^{T}\|X_0\|^2_{2,t}\mathrm{d}t + |g(0)|^2 + \left( \int_{0}^{T}|f_t^0|\mathrm{d}t \right)^2 \right] \\
	&\leq C\mathbb{E}\left[ |x|^2 + |g(0)|^2 + \left(\int_{0}^{T}|f_t^0| + |b_t^0|\mathrm{d}t\right)^2 + \int_{0}^{T}|\sigma_t^0|^2\mathrm{d}t \right] \\
	&= C(\mathbb{E}\left[ |x|^2 + |g(0)|^2 \right] + I^2_0),
	\end{align*}
	where the constants $ C $ may vary from line to line. Now let's examine the forward process $ X $. By standard estimates on SDEs (see e.g. \cite[Chapter 9]{Cvitanic+Zhang}), we get
	\begin{align*}
	\sup_{0\leq t\leq T}\mathbb{E}[\|X\|_{2,t}] 
	&\leq C\left(|x|^2 + \mathbb{E}\left[\left(\int_{0}^{T}(|b(t,0,Y_t,Z_t)|\mathrm{d}t\right)^2 + \int_{0}^{T}|\sigma(t,0,Y_t,Z_t)|^2\mathrm{d}t \right]\right) \\
	&\leq C\left(|x|^2 + \|(Y,Z)\|^2_2 + \mathbb{E}\left[\left(\int_{0}^{T}(|b_t^0|\mathrm{d}t\right)^2 + \int_{0}^{T}|\sigma_t^0|^2\mathrm{d}t \right]\right) \\
	&\leq C(\mathbb{E}\left[ |x|^2 + |g(0)|^2 \right] + I^2_0).
	\end{align*}
	Combining the above inequalities, we get
	\begin{equation*}
	\sup_{t\in[0,T]}\Big\{\mathbb{E}\Big[\|X\|^2_{2,t} + |Y_{t}|^{2} + \int_{t}^{T}|Z_{s}|^{2}\mathrm{d}s\Big]\Big\} \leq C(\mathbb{E}\left[ |x|^2 + |g(0)|^2 \right] + I^2_0).
	\end{equation*}
\end{proof}

\begin{thm}[Stability Property of path-dependent FBSDE]\label{thm: Stability}
	Assume that $ (b,\sigma,f,g) $ and $ (b',\sigma',f',g') $ satisfy the same condition (i.e. they belong to the same case discussed in the Section 4). Let $ T $ be a time horizon on which both FBSDE have a solution, denoted respectively $ \Upxi = (X,Y,Z) $ and $ \Upxi' = (X',Y',Z') $. For $ \phi=b,\sigma,f,g $, denote $ \Delta\phi := \phi - \phi' $. Let 
	\begin{equation*}
		\Delta I^2_0 := \mathbb{E}\left[ \left(\int_{0}^{T}\left|\Delta f_t(\Upxi'_t)\right| + \left|\Delta b_t(\Upxi'_t)\right|\mathrm{d}t\right)^2 + \int_{0}^{T}\left|\Delta\sigma_t(\Upxi'_t)\right|^2\mathrm{d}t \right].
	\end{equation*}
	Then, we have
	\begin{equation*}
	\sup_{t\in[0,T]}\Big\{\mathbb{E}\Big[\|\Delta X\|^2_{2,t} + |\Delta Y_{t}|^{2} + \int_{t}^{T}|\Delta Z_{s}|^{2}\mathrm{d}s\Big]\Big\} \leq C\big(|\Delta x|^2 + \mathbb{E}\left[ |\Delta g(X')|^2 \right] + \Delta I^2_0\big).
	\end{equation*}
\end{thm}
\begin{proof}[\textbf{Proof of Theorem \ref{thm: Stability}}]
	We follow the steps of the proof of Theorem 8.1 in the paper \cite{ma2015}. Using the notation described in the Section 4, we have
	\begin{align}
	\text{}\mathrm{d}\Delta X_{t} 
	&= (b_{x}\|\Delta X\|_{2,t}+b_{y}\Delta Y_{t}+\Tr(b_{z}\Delta Z_{t}) + \Delta b_t( \Upxi'(t))\mathrm{d}t \nonumber \\
	&\qquad\qquad\qquad + (\sigma_{x}\|\Delta X\|_{2,t}+\sigma_{y}\Delta Y_{t}+\Tr(\sigma_{z}\Delta Z_{t}) + \Delta \sigma_t(\Upxi'))\mathrm{d}W_{t}, \label{forwardComparison}
	\end{align}
	\begin{equation}\label{backwardComparison}
	\mathrm{d}\Delta Y_{t} = \int_{t}^{T}(f_{x}\|\Delta X\|_{2,t} + f_{y}\Delta Y_{t}+\Tr(f_{z}\Delta Z_{t}) + \Delta f_t(\Upxi'(t))\mathrm{d}t - \int_{t}^{T}\Delta Z_{s}\mathrm{d}W_{t}, 
	\end{equation}
	with initial condition $ \Delta X_0 = x - x' $ and terminal condition $ \Delta Y_T  = g_x\|\Delta X\|_{2,t} + \Delta g(X') $. 
	
	Since both FBSDE satisfy the same condition, which means there exists $ n\in\mathbb{N} $ and $ 0= t_0<\cdots< t_n=T $ such that on each small interval $ [t_i,t_{i+1}] $, Lemma \ref{estimateFBSDE} applies to both FBSDE, which means Lemma \ref{estimateFBSDE} applies equally to the above FBSDE \eqref{forwardComparison}-\eqref{backwardComparison}. Denote $ u_1 $ and $ u_2 $ the two associated decoupling fields. We have
	\begin{align}
	&\qquad\qquad\sup_{t\in[t_i,t_{i+1}]}\Big\{\mathbb{E}\Big[\|\Delta X\|^2_{2,t} + |\Delta Y_{t}|^{2} + \int_{t}^{t_{i+1}}|\Delta Z_{s}|^{2}\mathrm{d}s\Big]\Big\} \nonumber \\
	&\leq C\mathbb{E}\Bigg[ \|\Delta X\|^2_{2,t_i} + |\Delta u(t_{i+1}, X')|^2 + \left(\int_{t_i}^{t_{i+1}}|\Delta f_t(\Upxi'(t))| + |\Delta b_t(\Upxi'(t))|\mathrm{d}t\right)^2 + \int_{t_i}^{t_{i+1}}|\Delta \sigma_t(\Upxi'(t))|^2\mathrm{d}t \Bigg] \nonumber \\
	&\leq C(\mathbb{E}\left[ \|\Delta X\|^2_{2,t_i} + |\Delta u(t_{i+1}, X')|^2 \right] + \Delta I^2_0 ). \label{iteration}
	\end{align} 
	Apply Lemma \ref{estimateFBSDE} to the above linear forward-backward equation \eqref{forwardComparison}-\eqref{backwardComparison} with initial condition $ \Delta X_{\cdot\wedge t_i}=0 $ and terminal condition $ u_{x}\|\Delta X\|_{2,t_{i+1}} + \Delta u(t_{i+1}, X') $ on $ [t_i, t_{i+1}] $, and note that the difference between the solution of equation $ (b,\sigma,f,g) $ with initial condition $ X_{t_i} = X'_{t_{i}} $ and terminal condition $ Y_{t_{i+1}} = u(t_{i+1}, X) $ and the solution of equation $ (b',\sigma',f',g') $ with initial condition $ X_{t_i} = X'_{t_{i}} $ and terminal condition $ Y_{t_{i+1}} = u'(t_{i+1}, X) $ on the interval $ [t_i, t_{i+1}] $ is exactly the very solution, we get
	\begin{equation*}
	\mathbb{E}[|\Delta u(t_i, X')|^2] = \mathbb{E}[|\Delta Y_{t_i}|^2] \leq C\Delta I^2_0 + C\mathbb{E}[|\Delta u(t_{i+1}, X')|^2].
	\end{equation*}
	
	By iteration one can show that with a larger constant $ C $, we have 
	\begin{equation*}
	\mathbb{E}[|\Delta u(t_i, X')|^2] \leq C(\mathbb{E}\left[ |\Delta g(X')|^2 \right] + \Delta I^2_0).
	\end{equation*}
	Now apply again Lemma \ref{estimateFBSDE} but on the forward equation \eqref{forwardComparison} on the interval $ [t_i,t_{i+1}] $, together with the above inequalities, we get
	\begin{align*}
	\mathbb{E}\left[\|\Delta X\|^2_{2,t_{i+1}}\right] 
	&\leq C(\mathbb{E}\left[ \|\Delta X\|^2_{2,t_i} + |\Delta u(t_{i+1}, X')|^2 \right] + \Delta I^2_0 ) \\
	&\leq C(\mathbb{E}\left[ \|\Delta X\|^2_{2,t_i} + \Delta g(X')\right] + \Delta I^2_0 ).
	\end{align*}
	By iteration one can show that with a larger constant $ C $, we have
	\begin{equation*}
	\mathbb{E}\left[\|\Delta X\|^2_{2,t_i}\right]  \leq C(\mathbb{E}\left[ |\Delta x|^2 + \Delta g(X') \right] + \Delta I^2_0 ).
	\end{equation*}
	Injecting the above inequalities into the inequality \eqref{iteration}, we get
	\begin{equation*}
	\sup_{t\in[t_i,t_{i+1}]}\Big\{\mathbb{E}\Big[\|\Delta X\|^2_{2,t} + |\Delta Y_{t}|^{2} + \int_{t}^{t_{i+1}}|\Delta Z_{s}|^{2}\mathrm{d}s\Big]\Big\} \leq C(\mathbb{E}\left[ |\Delta x|^2 + |\Delta g(X')|^2 \right] + \Delta I^2_0 ). 
	\end{equation*}
	We conclude by summing up both side from $ i=0 $ to $ i=n $.
	
\end{proof}

\section{Technical proofs} \label{S4}
\begin{proof}[\textbf{Proof of Theorem \ref{GlobleExistence}}]
		We shall follow the steps of the proof of Theorem 2.3 in \cite{ma2015} of Ma, Wu, Zhang (Detao), Zhang (Jianfeng).
		
		(Existence) Let $ 0=t_{0}<t_{1}<\cdots<t_{n}=T $ be a partition of $ [0,T] $ such that $ \forall i\in\llbracket 1,n \rrbracket, t_{i+1} - t_{i}>\delta $. On $ [t_{0}, t_{1}] $, the FBSDE with initial value $ x $ and terminal value $ u(t_{1}, X) $ has an unique solution $ (X^{t_{0},t_{1}}, Y^{t_{0},t_{1}}, Z^{t_{0},t_{1}}) $ that satisfies $ Y^{t_{0},t_{1}}_{t} = u(t,X) $. On $ [t_{1}, t_{2}] $, the FBSDE with initial value $ X^{t_{0},t_{1}}_{\wedge t_{1}} $ and terminal value $ u_{t_{2}}(X) $ has an unique solution $ (X^{t_{1},t_{2}}, Y^{t_{1},t_{2}}, Z^{t_{1},t_{2}}) $ that satisfies again $ Y^{t_{1},t_{2}}_{t} = u(t, X) $. The initial condition of $ X^{t_{1}, t_{2}} $ is $ X^{t_{0}, t_{1}} $. By patching them together we obtain an forward process $ X^{t_{0},t_{2}} $, which can be used as initial value for the FBSDE on the interval $ [t_{2}, t_{3}] $. Repeating this procedure forwardly in time $ n $ times, we get a solution on each of the interval of the partition $ 0=t_{0}<t_{0}<\cdots<t_{n}=T $. 
		
		We notice that the forward process on $ [0, T] $ has been constructed during the above procedure. We only need to prove that the pieces of the backward process can be patched together. Notice that
		\begin{equation}\label{key}
		Y^{t_{i},t_{i+1}}_{t_{i}+} = Y^{t_{i},t_{i+1}}_{t_{i}} = u(t_{i}, X) = Y^{t_{i-1},t_{i}}_{t_{i}},
		\end{equation}
		which means the backward process $ Y $ defined on each interval $ [t_{i}, t_{i+1}] $ by the above procedure is continuous. Moreover, we have $ Y_{t} = u(t, X) $ and in particular, $ Y_{T} = g(X) $. One can check easily that $ (X,Y,Z) $ verifies the FBSDE with initial condition $ X_{0} = x $ and terminal condition $ Y_{T} = g(X) $.
		
		We can check easily that $ (X, Y, Z) $ is a solution of the FBSDE with initial value $ x $ and terminal value $ u(T, X) = g(X) $.
		
		(Uniqueness) Let $ (\tilde{X},\tilde{Y},\tilde{Z}) $ be another solution of the FBSDE with the same initial and terminal condition. By the definition of decoupling field, on the interval $ [t_{n-1}, t_{n}] $, we have $ \tilde{Y}_{t} = u(t, \tilde{X}) $. This implies that $ (\tilde{X},\tilde{Y},\tilde{Z}) $ satisfies the FBSDE with initial condition $ \tilde{X}_{\cdot\wedge t_{n-2}} $ on $ [t_{n-2}, t_{n-1}] $. Therefore, $ \tilde{Y}_{t} = u(t, \tilde{X}) $ is satisfied on $ [t_{n-2}, t_{n-1}] $. Repeating this procedure backwardly in time and we get $ \tilde{Y}_{t} = u(t, \tilde{X}) $ for $ t\in[0, T] $. 
		
		On $ [t_{0}, t_{1}] $, $ (\tilde{X},\tilde{Y},\tilde{Z}) $ satisfies the FBSDE with initial condition $ x $ and terminal condition $ \tilde{Y}_{t_{1}}=u(t_{1},\tilde{X}) $, by the uniqueness of solution, $ (X_{t}, Y_{t}, Z_{t}) = (\tilde{X}_{t}, \tilde{Y}_{t}, \tilde{Z}_{t}) $ on $ [t_{0}, t_{1}] $. In particular, the FBSDE on $ [t_{1}, t_{2}] $ has the same initial condition for $ X $ and $ \tilde{X} $. Repeating the arguments forwardly in time and we can see that $ (X_{t}, Y_{t}, Z_{t}) = (\tilde{X}_{t}, \tilde{Y}_{t}, \tilde{Z}_{t}) $ on $ [0,T] $. 	
\end{proof}

\begin{proof}[\textbf{Proof of Theorem \ref{thm_contraction}}]
	Let $ (y, z)\in\mathbb{H}^2_T\times\mathbb{H}^2_T $. Let $ (X,Y,Z) $ be the unique solution of the following decoupled FBSDE:
	\begin{numcases}
	\text{ }\mathrm{d}X_{t} = b_{t}(X, y_{t}, z_{t})\mathrm{d}t + \sigma_{t}(X, y_{t}, z_{t})\mathrm{d}W_{t} & $ X_{0} = x $ \label{decoupledForward} \\
	\mathrm{d}Y_{t} = -f_{t}(X, y_{t}, z_{t})\mathrm{d}t + Z_{t}\mathrm{d}W_{t} & $ Y_{T} = g(X_{\cdot\wedge T}) $.\label{decoupledBackward}
	\end{numcases}
	We can then define the following mapping $ (y,z)\in\mathbb{H}^2_T\times\mathbb{H}^2_T\mapsto(Y,Z)\in\mathbb{H}^2_T\times\mathbb{H}^2_T $. Our goal is to show that this mapping is a contraction for some norm that we shall define later. First of all, we notice that if the mapping is indeed a contraction, then the fixed point of the mapping $ (y,z) $ and the corresponding forward process $ X $ defined by the equation\eqref{decoupledForward} are a solution of the FBSDE \eqref{forward}-\eqref{backward}. Conversely, if $ (X,Y,Z) $ is a solution of the FBSDE \eqref{forward}-\eqref{backward}, then $ (Y,Z) $ is a fixed point of the mapping we define above.
	
	Let $ (y,z) $ and $ (y',z') $ be two pairs of progressively measurable processes and let $ (X,Y,Z) $ and $ (X',Y',Z') $ be the corresponding solutions of the above decoupled FBSDE \eqref{decoupledForward}-\eqref{decoupledBackward}. 
	
	Denote $ \Delta \alpha := \alpha - \alpha' $ for $ \alpha = y,z,X,Y,Z $ and denote 
	\begin{equation*}
		\Delta_{x}\xi_{t} := \xi_{t}(X,Y_t,Z_t) - \xi_t(X',Y_t,Z_t)
	\end{equation*}
	\begin{equation*}
		\Delta_{y}\xi_{t} := \xi_{t}(X',Y,Z) - \xi_t(X',Y',Z)
	\end{equation*}
	\begin{equation*}
		\Delta_{z}\xi_{t} := \xi_{t}(X',Y',Z) - \xi_t(X',Y',Z'),
	\end{equation*}
	for $ \xi = b, \sigma, f $. Clearly,
	\begin{equation*}
	\Delta X_{t} = \int_{0}^{t}(\Delta_{x}b_{s}+\Delta_{y}b_{s}+\Delta_{z}b_{s})\mathrm{d}s + \int_{0}^{t}(\Delta_{x}\sigma_{s}+\Delta_{y}\sigma_{s}+\Delta_{z}\sigma_{s})\mathrm{d}W_{s}.
	\end{equation*}
	
	By Ito's formula, we get
	\begin{equation*}
	\mathbb{E}[|\Delta X_{t}|^{2}] = \mathbb{E}\Big[\int_{0}^{t}2\Delta X_{s}(\Delta_{x}b_{s}+\Delta_{y}b_{s}+\Delta_{z}b_{s})\mathrm{d}s + \int_{0}^{t}\big|\Delta_{x}\sigma_{s}+\Delta_{y}\sigma_{s}+\Delta_{z}\sigma_{s}\big|^{2}\mathrm{d}s\Big].
	\end{equation*}
	
	By Cauchy-Schwarz Inequality and the inequality \eqref{LipschitzCondition}, we get
	\begin{equation*}
	2\Delta X_{s}\cdot \Delta_{x}b_{s}
	\leq 2K_{0}|\Delta X_{s}| \|\Delta X\|_{2,s} 
	\leq K_{0}\left(2|X_{s}|^{2} + \int_{0}^{t}|\Delta X_{s}|^{2}\mathrm{d}s\right);
	\end{equation*}
	
	\begin{equation*}
	2\Delta X_{s}\cdot\Delta_{y}b_{s}
	\leq 2K_{0}|\Delta X_{s}||\Delta y_{s}| 
	\leq K_{0}(|\Delta X_{s}|^{2} + |\Delta y_{s}|^{2});
	\end{equation*}
	
	\begin{equation*}
	2\Delta X_{s}\cdot\Delta_{z}b_{s}
	\leq 2K_{0}|\Delta X_{s}||\Delta z_{s}| 
	\leq K_{0}\left(\frac{|\Delta X_{s}|^{2}}{\epsilon} + \epsilon|\Delta z_{s}|^{2}\right).
	\end{equation*}
	Combing the above inequalities, we get
	\begin{equation}\label{Ineq1}
	\int_{0}^{t}2\Delta X_{s}(\Delta_{x}b_{s}+\Delta_{y}b_{s}+\Delta_{z}b_{s})\mathrm{d}s \leq \int_{0}^{t}K_{0}((3+t+\epsilon^{-1})|\Delta X_{s}|^{2}+|\Delta y_{s}|^{2}+\epsilon|\Delta z_{s}|^{2})\mathrm{d}s.
	\end{equation}
	
	Using Minkowski inequality and arithmetic-geometric inequality, we get
	\begin{align}
	&\qquad |\Delta_{x}\sigma_{s}+\Delta_{y}\sigma_{s}+\Delta_{z}\sigma_{s}|^{2} \nonumber \\
	&\leq (K_{0}\|\Delta X\|_{2,s} + K_{0}|\Delta y_{s}| + |\nabla_z\sigma|_\infty|\Delta z_{s}|)^{2} \nonumber \\
	&\leq 2K^{2}_{0}\left(1+\frac{K_{0}}{\epsilon}\right)(\|\Delta X\|^{2}_{2,s} + |\Delta y_{s}|^{2}) + (|\nabla_z\sigma|_\infty^2 + K_{0}\epsilon)|\Delta z_{s}|^{2} \nonumber \\
	&\leq 2K^{2}_{0}\left(1+\frac{K_{0}}{\epsilon}\right)(|\Delta X_{s}|^{2} + \int_{0}^{t}|\Delta X_{s}|^{2}\mathrm{d}s + |\Delta y_{s}|^{2}) + (|\nabla_z\sigma|_\infty^2 + K_{0}\epsilon)|\Delta z_{s}|^{2}. \label{Ineq2}
	\end{align}	
	
	Combining the inequality \eqref{Ineq1} and \eqref{Ineq2}, we get
	\begin{align*}
	&\qquad \mathbb{E}\Big[\int_{0}^{t}2\Delta X_{s}(\Delta_{x}b_{s}+\Delta_{y}b_{s}+\Delta_{z}b_{s})\mathrm{d}s + \int_{0}^{t}\big|\Delta_{x}\sigma_{s}+\Delta_{y}\sigma_{s}+\Delta_{z}\sigma_{s}\big|^{2}\mathrm{d}s\Big] \\
	&\leq \int_{0}^{t}C_{\epsilon}(|\Delta X_{s}|^{2}+|\Delta y_{s}|^{2})+(2K_{0}\epsilon + |\nabla_z\sigma|_\infty^2)|\Delta z_{s}|^{2}\mathrm{d}s,
	\end{align*}
	where
	\begin{equation*}
	C_{\epsilon} := 2K^{2}_{0}\left(1+\frac{K_{0}}{\epsilon}\right)(1+T) + K_{0}(3+T+\epsilon^{-1}).
	\end{equation*}
	
	By Gronwall Inequality, we get
	\begin{align*}
	\mathbb{E}\big[|\Delta X_{t}|^{2}\big] &\leq \mathbb{E}\Big[e^{C_{\epsilon}t}\int_{0}^{t}(C_{\epsilon}|\Delta y_{s}|^{2} + (2K_{0}\epsilon + |\nabla_z\sigma|_\infty^2)|\Delta z_{s}|^{2})\mathrm{d}s \Big] \\
	&\leq \mathbb{E}\Big[e^{C_{\epsilon}T}\int_{0}^{T}(C_{\epsilon}|\Delta y_{s}|^{2} + (2K_{0}\epsilon + |\nabla_z\sigma|_\infty^2)|\Delta z_{s}|^{2})\mathrm{d}s \Big].
	\end{align*}
	Again by applying the Ito's Formula to the process $ |\Delta Y_{t}|^{2} $, we get
	\begin{equation}\label{Y_equation}
	\mathbb{E}\big[|\Delta Y_{t}|^{2} + \int_{t}^{T}|\Delta Z_{s}|^{2}\big] = |\Delta Y_{T}|^{2} + \mathbb{E}\Big[\int_{t}^{T}2\Delta Y_{s}(\Delta_{x}f_{s} + \Delta_{y}f_{s} + \Delta_{z}f_{s})\mathrm{d}s\Big].
	\end{equation}
	By Cauchy-Schwarz inequality and arithmetic-geometric inequality, we get
	\begin{align*}
	\qquad 2\Delta Y_{s}(\Delta_{x}f_{s} + \Delta_{y}f_{s} + \Delta_{z}f_{s}) 
	&\leq 2K_{0}|\Delta Y_{s}|(\|\Delta X\|_{2,s} + |\Delta y_{s}| + |\Delta z_{s}|) \\
	&\leq K_{0}\big( (2+\epsilon^{-1})|\Delta Y_{s}|^{2} + \|\Delta X\|^{2}_{2,s} + |\Delta y_{s}|^{2} + \epsilon|\Delta z_{s}|^{2} \big).
	\end{align*}
	
	Combining equation \eqref{Y_equation} and the above inequality, we get
	\begin{align}
	&\qquad\mathbb{E}\big[|\Delta Y_{t}|^{2} + \int_{t}^{T}|\Delta Z_{s}|^{2}\big] \nonumber	\\
	&\leq |\Delta Y_{T}|^{2} + K_{0}\mathbb{E}\Big[\int_{t}^{T}(2+\epsilon^{-1})|\Delta Y_{s}|^{2}+\epsilon|\Delta z_{s}|^{2}+|\Delta y_{s}|^{2} \mathrm{d}s+ (T+1)\int_{0}^{T}|\Delta X_{s}|^{2}\mathrm{d}s \Big] \nonumber \\	
	&\leq |\Delta Y_{T}|^{2} + \mathbb{E}\Big[\int_{t}^{T}\tilde{C}_{\epsilon}|\Delta Y_{s}|^{2}\mathrm{d}s + \int_{0}^{T}K_{0}(\epsilon+T(T+1)e^{C_{\epsilon}T}(2K_{0}\epsilon+|\nabla_z\sigma|_\infty^2))|\Delta z_{s}|^{2} \nonumber \\
	&\qquad +(K_{0}+T(T+1)K_{0}e^{C_{\epsilon}T}C_{\epsilon})|\Delta y_{s}|^{2}\mathrm{d}s\Big] \nonumber \\		
	&\leq \mathbb{E}\Big[K^{2}_{1}\int_{0}^{T}|\Delta X_{t}|^{2}\mathrm{d}t + K^{2}_{1}|\Delta X_{T}|^{2}\Big] + \mathbb{E}\Big[\int_{t}^{T}\tilde{C}_{\epsilon}|\Delta Y_{s}|^{2}\mathrm{d}s \nonumber \\
	&\qquad + \int_{0}^{T}K_{0}(\epsilon+T(T+1)e^{C_{\epsilon}T}(2K_{0}\epsilon+|\nabla_z\sigma|_\infty^2))|\Delta z_{s}|^{2}+(K_{0}+T(T+1)K_{0}e^{C_{\epsilon}T}C_{\epsilon})|\Delta y_{s}|^{2}\mathrm{d}s\Big] \nonumber \\		
	&\leq \mathbb{E}\Big[\int_{0}^{T}(T+1)K^{2}_{1}e^{C_{\epsilon}T}(C_{\epsilon}|\Delta y_{s}|^{2} + (2K_{0}\epsilon + |\nabla_z\sigma|_\infty^2)|\Delta z_{s}|^{2})\mathrm{d}s + \int_{t}^{T}\tilde{C}_{\epsilon}|\Delta Y_{s}|^{2}\mathrm{d}s \nonumber \\
	&\qquad + \int_{0}^{T}K_{0}e^{C_{\epsilon}T}(\epsilon+T(T+1)(2K_{0}\epsilon+|\nabla_z\sigma|_\infty^2))|\Delta z_{s}|^{2} + (K_{0}+T(T+1)K_{0}e^{C_{\epsilon}T}C_{\epsilon})|\Delta y_{s}|^{2}\mathrm{d}s\Big] \nonumber \\		
	&\leq \mathbb{E}\Big[\int_{t}^{T}\tilde{C}_{\epsilon}|\Delta Y_{s}|^{2} + \int_{0}^{T}C_{y}(T,\epsilon)|\Delta y_{s}|^{2} + C_{z}(T, \epsilon)|\Delta z_{s}|^{2}\mathrm{d}s\Big]. \label{eq1}
	\end{align} 
	In the last line of the above inequalities, the constants $ \tilde{C}_{\epsilon} $, $ C_{y}(T,\epsilon) $ and $ C_{z}(T, \epsilon) $ are given by
	\begin{equation*}
	\tilde{C}_{\epsilon} := K_{0}(2+\epsilon^{-1})\text{, }
	C_{y}(T,\epsilon) := (T+1)K^{2}_{1}e^{C_{\epsilon}T}C_{\epsilon}+K_{0}+T(T+1)K_{0}e^{C_{\epsilon}T}C_{\epsilon}
	\end{equation*}
	and
	\begin{equation*}
	C_{z}(T,\epsilon) := (T+1)K^{2}_{1}e^{C_{\epsilon}T}(2K_{0}\epsilon + |\nabla_z\sigma|_\infty^2) + K_{0}(\epsilon+T(T+1)e^{C_{\epsilon}T}(3K_{0}\epsilon+|\nabla_z\sigma|_\infty^2)).
	\end{equation*}	
	Using Gronwall's inequality on $ Y $, we get
	\begin{align}
	\mathbb{E}\big[|\Delta Y_{t}|^{2}\big]
	&\leq\mathbb{E}\Big[e^{\tilde{C}_{\epsilon}(T-t)}\int_{0}^{T}C_{y}(T,\epsilon)|\Delta y_{s}|^{2} + C_{z}(T, \epsilon)|\Delta z_{s}|^{2}\mathrm{d}s\Big] \nonumber \\
	&\leq\mathbb{E}\Big[e^{\tilde{C}_{\epsilon}T}\int_{0}^{T}C_{y}(T,\epsilon)|\Delta y_{s}|^{2} + C_{z}(T, \epsilon)|\Delta z_{s}|^{2}\mathrm{d}s\Big]. \label{eq2}
	\end{align}
	Plug the inequality\eqref{eq2} into the inequality\eqref{eq1}, we get, for all $ t\in[0, T]: $
	\begin{align*}
	\mathbb{E}\big[|\Delta Y_{t}|^{2} + \int_{t}^{T}|\Delta Z_{s}|^{2}\big] 
	&\leq\mathbb{E}\Big[(\tilde{C}_{\epsilon}Te^{\tilde{C}_{\epsilon}T}+1)\int_{0}^{T}C_{y}(T,\epsilon)|\Delta y_{s}|^{2} + C_{z}(T, \epsilon)|\Delta z_{s}|^{2}\mathrm{d}s\Big] \\
	&\leq \Big(T(\tilde{C}_{\epsilon}Te^{\tilde{C}_{\epsilon}T}+1)C_{y}(T,\epsilon)+ (\tilde{C}_{\epsilon}Te^{\tilde{C}_{\epsilon}T}+1)C_{z}(T, \epsilon)\Big) \\
	&\qquad \times\sup_{t\in[0,T]}\Big\{\mathbb{E}\Big[|\Delta y_{t}|^{2} + \int_{t}^{T}|\Delta z_{s}|^{2}\mathrm{d}s\Big]\Big\}.
	\end{align*}
	Denote $ \gamma(\epsilon, T):= T(\tilde{C}_{\epsilon}Te^{\tilde{C}_{\epsilon}T}+1)C_{y}(T,\epsilon)+ (\tilde{C}_{\epsilon}Te^{\tilde{C}_{\epsilon}T}+1)C_{z}(T, \epsilon) $ and
	\begin{equation*}
	\|(y,z)\|^2_2 := \sup_{t\in[0,T]}\Big\{\mathbb{E}\Big[|y_{t}|^{2} + \int_{t}^{T}|z_{s}|^{2}\mathrm{d}s\Big]\Big\}.
	\end{equation*}
	For a fixed $ \epsilon>0 $, $ \gamma(\epsilon, T) $ converges to $ K^{2}_{1}(2K_{0}\epsilon+|\nabla_z\sigma|_\infty^2) + K_{0}\epsilon $ while $ T\to0 $. Since $ K_{1}|\nabla_z\sigma|_\infty<1 $, we can find a pair of $ \epsilon $ and $ T $ small enough such that $ \gamma(\epsilon, T)<1 $, in which case, the mapping $ (y,z)\mapsto(Y,Z) $ is a contraction. Denote $ \mathcal{L} $ the space of all $ \mathbb{F} $-adapted processes $ (Y,Z) $ such that $ \|(Y,Z)\|_2<+\infty $. We can show easily that the space $ (\mathcal{L}, \|\cdot\|_2) $ is a Banach space, by the contraction mapping theorem, the mapping $ (y,z)\mapsto(Y,Z) $ has an unique fixed point $ (Y,Z) $.
\end{proof}

\bibliographystyle{plain}
\bibliography{Bibliographie} 

\end{document}